\documentclass[10pt]{amsart}
\usepackage{a4wide}
\usepackage{amsfonts}
\usepackage{psfrag}
\usepackage{color}
\usepackage[all]{xypic}
\usepackage{hyperref}
\usepackage{amsmath}
\usepackage{amssymb}
\usepackage{amsthm}
\usepackage{graphicx}
\usepackage{mathtools}

\newtheorem{theorem}{Theorem}[section]

\newtheorem{lemma}[theorem]{Lemma}
\newtheorem{proposition}[theorem]{Proposition}

\theoremstyle{definition}
\newtheorem{definition}[theorem]{Definition}
\newtheorem{remark}[theorem]{Remark}

\newcommand{\wY}{\widetilde{Y}}

\newcommand{\Ss}{\mathbb{S}}
\newcommand{\PP}{\mathbb{P}}
\newcommand{\PPP}{\mathbf{P}}
\newcommand{\Hh}{\mathcal{H}}
\newcommand{\HH}{\mathbb{H}}

\newcommand{\D}{\mathrm{d}}
\newcommand{\UU}{\mathbb{U}}

\newcommand{\R}{\mathbb{R}}
\newcommand{\Ff}{\mathcal{F}}

\newcommand{\rrho}{\overline{\rho}}

\numberwithin{equation}{section}

\begin{document}
\title{On the probability of hitting the boundary for Brownian motions on the SABR plane}

\author{Archil Gulisashvili}
\address{Department of Mathematics, Ohio University}
\email{gulisash@ohio.edu}

\author{Blanka Horvath}
\address{Department of Mathematics, Imperial College London}
\email{b.horvath@imperial.ac.uk}

\author{Antoine Jacquier}
\address{Department of Mathematics, Imperial College London}
\email{a.jacquier@imperial.ac.uk}
\date{\today}

\thanks{
BH acknowledges financial support from the SNF Early Postdoc Mobility Grant 165248.
AJ acknowledges financial support from the EPSRC First Grant EP/M008436/1.}
\maketitle

\begin{abstract}
Starting from the hyperbolic Brownian motion as a time-changed Brownian motion, 
we explore a set of probabilistic models--related to the SABR model in mathematical finance--which 
can be obtained by geometry-preserving transformations,
and show how to translate the properties of the hyperbolic Brownian motion (density, probability mass, drift) 
to each particular model.
Our main result is an explicit expression for the probability of any of these models 
hitting the boundary of their domains, 
the proof of which relies on the properties of the aforementioned transformations 
as well as time-change methods.
\end{abstract}

\section{Introduction}
Stochastic analysis on manifolds is a vibrant and well-studied field dating back to the seminal work 
of Varadhan~\cite{Varadhan}, followed by Elworthy~\cite{Elworthy}, 
Hsu~\cite{Hsu}, Stroock~\cite{Stroock}, Grigoryan~\cite{Grigoryan}, Avramidi~\cite{Avramidi} and, 
in a financial context ~\cite{Armstrong, GatheralHsu, HLW, Henry-Labordere};
of particular importance in these works is Brownian motion on a Riemannian manifold\footnote{Recall that a Markov process $X$ with state space $M$ is a Brownian motion on a Riemannian manifold~$(M,g)$ if 
its law solves the martingale problem corresponding to the Laplace-Beltrami operator of~$(M,g)$;
see~\cite{Grigoryan, Hsu}.}.
The underlying manifold here is the state space of the process, 
which is in most cases a complete open manifold. 
This is not the case, for example for the following process:
\begin{equation}\label{eq:SABRDrift}
\begin{array}{rlrl}
\D X_t & = \displaystyle Y_t X_t^{\beta}\D W_t + \frac{\beta}{2} Y_t^2 X_t^{2\beta-1}\D t, 
\qquad & X_0 & = x_0> 0,\\
\D Y_t & = \nu Y_t \D Z_t, \qquad & Y_0 & = y_0>0,\\
\D \langle Z,W\rangle_t & = \rho\,\D t,
\end{array}
\end{equation}
where $\nu>0$, $\rho\in (-1,1)$, $\beta \in [0,1)$,
and $W$ and $Z$ are two correlated Brownian motions on a filtered probability space
$(\Omega, \Ff, (\Ff_t)_{t\geq 0}, \PP)$.
The case $\beta=1$ is excluded of the analysis as it is a trivial case (see Remark~\ref{R:rem0}).
While in the case $\beta=0$, the natural state space is
$\Hh:=\R\times(0,\infty)$ both open and complete, 
the natural underlying space when $\beta>0$
is $\Hh_+:=[0,\infty)\times(0,\infty)$, a (non-complete) manifold with boundary 
$\{0\}\times(0,\infty)$. 
In these situations it is natural to wonder about the probability that 
the process on this state space never reaches the boundary. 
In the specific case $\beta=\rho=0$, $\nu=1$, the SDE~\eqref{eq:SABRDrift} 
describes the dynamics of Brownian motion on hyperbolic plane.
This (hyperbolic) Brownian motion is particularly tractable, and its density is known in closed form. Therefore, it is a good starting point for the study of the law and the large-time behaviour of processes of the form \eqref{eq:SABRDrift}.
Indeed, restricting hyperbolic Brownian motion to $\Hh_+$ with the addition of Dirichlet boundary conditions along the ray 
$\{0\}\times(0,\infty)$
makes this process
suitable for the framework of Hobson~\cite[Theorem 4.2]{Hobson}, 
who studies the large-time behaviour of stochastic volatility models via coupling and comparison methods. 
There, Hobson provides the following classification (and examples) of the large-time behaviour of sample paths 
of the~$X$ process for such models: 
(i) it can hit zero in finite time, (ii) it converges to a strictly positive limit, 
or~(iii) it is always positive, but converges to zero as time tends to infinity.
Note that these cases are not necessarily exclusive from one another, and (i) and (ii) can both happen with positive probability;
for a given model, however, it is in general difficult to estimate these probabilities precisely.
In this article we single out some processes for which these probabilities can be quantified. 
The hyperbolic Brownian motion, for example, exhibits such a non-trivial large time behaviour,
where both~(i) and~(ii) occur with strictly positive probabilities, 
for which we derive explicit expressions using time change techniques and properties of hitting times of Brownian motion. 
We further present transformations of the hyperbolic Brownian motion 
under which this large-time property remains valid,
and provide formulae for these probabilities;
the resulting processes turn out to be precisely of the form~\eqref{eq:SABRDrift}\footnote{
Up to a deterministic time change, $\nu$ can be taken equal to one, 
and we assume this without loss of generality.}.
One particular feature of~\eqref{eq:SABRDrift} is that the state space is not compact.
In geometry, the large-time behaviour of the heat kernel (and the corresponding semigroup) 
is well known in the compact case--via 
its infinite series representation (see for example~\cite{ChavelBook})--and 
several results have extended this to the non-compact case
(see for example~\cite{Chavel, Li, Pinchover1, Pinchover2}). The majority of the existing literature however focuses on the case where the state space is a complete manifold, and results for the case of manifolds with boundary are rare~\cite{Vasilievic}.
In probability, such results, 
known for continuous time Markov chains on finite state space (by Perron-Frobenius theorem), 
do not have a general formulation for infinite state space.

In this article, we display several tractable properties of the solution to~\eqref{eq:SABRDrift},
which we refer to as a Brownian motion on the SABR plane,
since it characterises a Brownian motion in a suitably chosen Riemannian manifold with boundary 
(the SABR plane cf.~\cite[Subsection 3.2]{HLW}), 
as emphasized in Lemma~\ref{Lem:varphi} below.
We analyse furthermore the effect of the parameters $\beta$ and $\rho$ on the large-time behaviour of the process~\eqref{eq:SABRDrift} by focussing on the cases where either one of these parameters (or both) is zero. Our analysis confirms that the large-time behaviour is independent of the order in which the parameters $\beta$ and $\rho$ were introduced (this follows from the commutativity of the diagram 
on Page~\pageref{diagram}, 
proved in Theorem~\ref{Th:MaintheoremSec2}). 
We also relate (whenever possible) the properties of this model to those of the SABR model
\begin{equation}\label{eq:SABRSDE}
\begin{array}{rlrl}
\D X_t & = \displaystyle Y_t X_t^{\beta}\D W_t, 
\qquad & X_0 & = x_0> 0,\\
\D Y_t & = \nu Y_t \D Z_t, \qquad & Y_0 & = y_0>0,\\
\D \langle Z,W\rangle_t & = \rho\,\D t,
\end{array}
\end{equation}
introduced in~\cite{ManagingSmileRisk, HLW}, and now widely used in financial markets.
Compared to the SABR model~\eqref{eq:SABRSDE}, 
the $X$-dynamics of~\eqref{eq:SABRDrift} include an additional drift term,
which appears in an expansion of the density only as a higher-order term perturbation correction~\cite{HLW}.
The behaviour of the drift in~\eqref{eq:SABRDrift}
is significantly different when $\beta \in (0,1/2)$ and $\beta\in(1/2, 1)$: 
in the former case the drift explodes when~$X$ approaches zero, while it vanishes in the latter case;
when $\beta=1/2$, the drift does not depend on~$X$ at all.
Interestingly however, the large-time behaviour remains invariant 
under some transformations affecting~$\beta$, 
while local properties (such as the density) can be translated from one case to another, 
reflecting the `phase transition' occurring in the above three cases. 
As observed in~\cite{DoeringHorvathTeichmann}, the constraints $\rho=0$ or $\beta=0$ 
are the only parameter configurations where certain advantageous regularity properties of~\eqref{eq:SABRSDE} are valid. 
In fact these are the only cases for which~\eqref{eq:SABRSDE} can be written 
as a Brownian motion on some weighted\footnote{See~\cite[Definition 3.17]{Grigoryan} 
for a precise definition of a weighted manifold.} manifold.
Note furthermore, that in the  $\beta=0$ case, the drift in~\eqref{eq:SABRDrift} vanishes and
the SDEs~\eqref{eq:SABRDrift} and~\eqref{eq:SABRSDE} coincide for all values of~$\rho$.
According to~\cite{BallandTran} and~\cite{HLW}, in the prevalence of low interest rates, 
the choice $\beta=0$ is rather common practice on interest rate desks, 
and, in this case, \eqref{eq:SABRSDE} is usually referred to as the `normal SABR' model.

Case~(i) in Hobson's classification coincides with the probability 
\begin{equation}\label{eq:PMass}
\PPP := \PP(X_t=0\text{ for some }t \in (0,\infty)),
\end{equation}
and our main result (Theorem~\ref{thm:Main}) is an exact expression for this probability as
$$
\PPP=\int_0^{\infty}\D t\int_0^t f(s,t)\D s,
$$
where the function~$f$ is available in closed form (as an infinite series).
In the case $\beta=\rho=0$, the function~$f$ admits the simplified formulation
($\mathrm{I}_{n}$ denoting the Bessel function of the first kind)
$$
f(s, t) = \frac{2\exp\left(-\frac{(x_0^2+y_0^2)(t+s)}{4st}\right)}{\pi (t-s)\sqrt{st}}
\sum_{n=1}^{\infty}n\sin\left[2n\left(\frac{\pi}{2} - \arctan\left(\frac{y_0}{x_0}\right)\right)\right]\mathrm{I}_{n}
\left(\frac{(x_0^2+y_0^2)(t-s)}{4st}\right).
$$
Similar probabilities (yet not this one in particular), of hitting some boundary, or a ball around it, 
have been studied for the hyperbolic Brownian motion in~\cite{Byczkowski, Macci, GriCoste, Lao, Lalley}

In Section~\ref{sec:Mappings}, we analyse the dynamics of~\eqref{eq:SABRDrift} under different parameter configurations,
and propose several space transformations to translate properties of one model configuration to the other.
In Section~\ref{SS:ltb}, we use these maps to derive an exact formula for~$\PPP$ for general parameter values.
We recall in Appendix~\ref{app:Geometry} some notions on the heat equation on manifolds, 
needed along the paper.
\vspace{3mm}

\textbf{Notations:} 
For a given real-valued stochastic process $X$ (with continuous paths) and a real number $z$, 
we denote by $\tau_z^X:=\inf\{t\geq 0: X_t=z\}$ the first hitting time of $X$ at level~$z$.
For convenience, we shall use the (now fairly standard) notation 
$\rrho:=\sqrt{1-\rho^2}$. For two functions $f$ and $g$, we shall write $f(z)\sim g(z)$ 
as~$z$ tends to zero
whenever $\lim\limits_{z\to 0} f(z)/g(z) = 1$.


\section{SABR geometry and geometry preserving mappings}\label{sec:Mappings}
We first exhibit a set of mappings allowing to translate the properties of one model configuration to another.
Let $\Hh:=\R\times (0,\infty)$ and $\Hh_+:= (0,\infty)^2$, and introduce the following pairs of spaces together with their metrics:
$$
\HH:=(\Hh, h),\quad
\HH_+:=(\Hh_+, h),\quad
\UU:=(\Hh_+, u),\quad
\Ss:=(\Hh_+, g),\quad
\Ss^0:=(\Hh, g^0),\quad
\Ss^0_+:=(\Hh_+, g^0),
$$
where the Riemannian metrics on their respective spaces are given by
\begin{equation*}
\begin{array}{rll}
h(\widetilde{x}, \widetilde{y})
 & \displaystyle = \frac{\D\widetilde{x}^2 + \D\widetilde{y}^2}{\widetilde{y}^2},
 & \displaystyle (\widetilde{x}, \widetilde{y}) \in \Hh,\\
g(x,y)
 & \displaystyle = \frac{1}{\rrho^2}\left(\frac{\D x^2}{y^2x^{2\beta} } - \frac{2 \rho \D x\D y}{y^2 x^{\beta}}
 + \frac{\D y^2}{y^2}\right),
 & \displaystyle (x,y) \in \Hh_+,\\
g^{0}(\widehat x,\widehat y)
 & \displaystyle = \frac{1}{\rrho^2}\left(\frac{\D \widehat x^2}{\widehat y^2}
  - \frac{2 \rho \D \widehat x\D \widehat y}{\widehat y^2}
 + \frac{\D \widehat y^2}{\widehat y^2}\right),
 & \displaystyle (\widehat{x}, \widehat{y}) \in \Hh,\\
u(\bar{x},\bar{y})
 & \displaystyle = \frac{1}{\bar{y}^2}\left(\frac{\D\bar{x}^2}{\bar{x}^{2\beta}}  + \D\bar{y}^2\right)
 & \displaystyle (\bar x,\bar y) \in \Hh_+.
\end{array}
\end{equation*}
Clearly, $\UU$ corresponds to the uncorrelated ($\rho=0$) model,
while $\Ss^0$ is the general SABR plane with $\beta=0$;
$\HH$ represents the classical Poincar\'e plane with its associated Riemannian metric~\cite[Section 3.9]{Grigoryan}, 
and~$\Ss$ the general SABR plane, generated by~\eqref{eq:SABRDrift}.
The following diagram summarises the different relations between the mappings and the spaces
(we also include the corresponding coordinate notations):

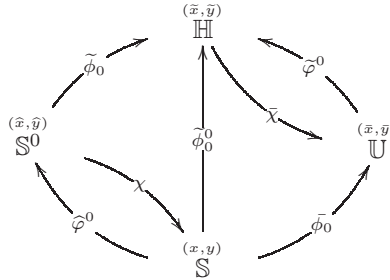
\centerline{\xymatrix{
& \quad\stackrel{_{(\widetilde{x},\widetilde{y})}}{\HH}\quad
\ar@/_1pc/[rd]|{\bar \chi}
 &
\\
\quad\stackrel{_{(\widehat x,\widehat y)}}{\Ss^{0}}\quad
\ar@/^1pc/[ur]|-{\widetilde \phi_{0}}
 \ar@/^1pc/[dr]|-{\chi}
 & 
 & \quad\stackrel{_{(\bar{x},\bar{y})}}{\mathbb{U}}\quad
 \ar@/_1pc/[ul]|-{\widetilde\varphi^{0}}
\\
& \quad\stackrel{_{(x,y)}}{\Ss}\quad
 \ar@/_1pc/[ur]|-{\bar{\phi_{0}}
}
 \ar@/^1pc/[ul]|-{\widehat \varphi^{0}}
 \ar[uu]|-{\widetilde \phi_{0}^{0}}
 & 
}\label{diagram}}
Regarding the mapping notations, subscripts~$_0$ are related to the correlation parameter 
(for example, the parameter $\rho$ vanishes by the action of~$\bar\phi_0$),
whereas superscripts~$^0$ indicate that the parameter~$\beta$ vanishes;
the map~$\chi$ reintroduces the parameter~$\beta$.
The mappings between these spaces are defined as follows:
\begin{equation}\label{eq:Isometries}
\begin{array}{lrlll}
\widetilde \phi_{0}^{0}: & \Ss \ni (x,y) & \longmapsto \left(\widetilde{x},\widetilde{y}\right)
 :=
\displaystyle \left(\frac{x^{1-\beta}}{\rrho(1-\beta)} - \frac{\rho y}{\rrho}, y \right)
 & \in \HH, \\
\widehat \varphi^{0}: & \Ss \ni (x,y) & \longmapsto \left(\widehat{x},\widehat{y}\right)
 :=
\displaystyle \left( \frac{x^{1-\beta}}{1-\beta}, y \right)
 & \in \Ss^{0},
\\
\bar{\phi}_{0} : & \Ss \ni (x,y) & \longmapsto \displaystyle (\bar{x},\bar{y}) := \left( (1-\beta)^{\frac{1}{1-\beta}}\left(\frac{ x^{1-\beta}}{\rrho(1-\beta)}-\frac{\rho y}{\rrho}\right)^{\frac{1}{1-\beta}}, y\right)
 & \in \UU, 
\quad \rho\leq 0\\
\widetilde \phi_{0}: & \Ss^{0}_+ \ni (\widehat x,\widehat y) & \longmapsto \left(\widetilde{x},\widetilde{y}\right)
 :=
\displaystyle \left( \frac{ \widehat x - \rho \widehat y}{\rrho}, \widehat y \right) & \in \HH, \\
\chi: & \Ss^{0}_+ \ni \left(\widehat{x},\widehat{y}\right) & \longmapsto \left(x,y\right)
 :=
\displaystyle \left( (1-\beta)^{\frac{1}{1-\beta}} \widehat{x}^{\frac{1}{1-\beta}}, \widehat{y} \right)
 & \in \Ss,\\
\bar \chi: & \HH_+ \ni \left(\widetilde{x},\widetilde{y}\right) & \longmapsto \displaystyle \left(\bar{x},\bar{y}\right)
 :=
\displaystyle \left( (1-\beta)^{\frac{1}{1-\beta}}\widetilde{x}^{\frac{1}{1-\beta}} , \widetilde{y} \right)
 & \in \UU,\\
\widetilde \varphi^{0}: & \UU \ni (\bar x,\bar y) & \longmapsto \displaystyle\left(\widetilde{x},\widetilde{y}\right)
 :=
\displaystyle \left( \frac{\bar{x}^{1-\beta}}{1-\beta}, \bar y \right)
 &  \in \HH_+, \\
\end{array}
\end{equation}
From now on, if not indicated otherwise, we restrict the domains of the above maps 
to the first quadrant~$\Hh_+$, 
which--when considering compositions--impose restrictions on the parameters 
in order to ensure that images also belong to this set 
(for example the restriction~$\rho \in (-1,0]$ needs to be imposed for the composition~$\widetilde\phi_0^0\circ\chi$).
While the map~$\widetilde\phi_0$ can be extended to the whole upper halfplane~$\Hh$, 
thus describing an asset with negative value, 
the maps~$\widetilde\varphi^0$, $\widetilde\phi_0^0$, $\chi$ and~$\bar\chi$ 
cannot be defined in the real plane.
They can be extended to the line $\{(x,y) \in \Hh: x=0\}$ though, and are non-differentiable there.
The following theorem gathers the properties of all these maps:

\begin{theorem}\label{Th:MaintheoremSec2}
The diagram is commutative and all the mappings in~\eqref{eq:Isometries} are local isometries on their respective spaces:
\begin{itemize}
\item the maps $\widehat\varphi^{0}$ and $\chi$ (resp. $\widetilde\varphi^{0}$ and $\bar\chi$)
 on~$\Hh_+$ are onto and inverse to one another;
\item the compositions $\bar \phi_0 \circ \widetilde \varphi^0$ and $\widehat \varphi^{0}\circ \widetilde \phi_{0}$ 
coincide with $\widetilde \phi_{0}^{0}$;
\item the equalities $\chi\circ\widetilde \phi_0^0=\widetilde \phi_0$ and $\widetilde \phi_0^0\circ \bar \chi=\bar{\phi}_{0}$ hold, 
and the latter is well defined for $\rho \in (-1,0]$;
\item the map $\widehat \varphi^{0}$ (resp. $\widetilde \varphi^0$) transforms the Brownian motion 
on~$(\Ss,g)$ (resp.~$(\UU,u)$) into the SABR model~\eqref{eq:SABRSDE} with $\beta=0$ 
(resp. $\rho=\beta=0$), which in turn is transformed back to Brownian motion on its original spaces 
by the map~$\chi$ (resp.~$\bar{\chi}$);
\item the maps $\bar \phi_{0}$ (resp. the extension of $\widetilde \phi_{0}$) transforms the Brownian motion on~$(\Ss,g)$ (resp. $(\Ss^0,g^0)$) into its uncorrelated version on~$(\UU,u)$ (resp. $(\HH,h)$).
\end{itemize}
\end{theorem}
\begin{proof}
The first three items follow from simple computations;
the remaining statements follow from Lemmas~\ref{Lem:varphi}, \ref{CorrHypBM}, \ref{Lem:SABRIsoS0}
and~\ref{Lem:SABRIsoU} below.
\end{proof}

\begin{remark}
The map~$\widetilde \phi_0^0$ was first considered in~\cite{HLW}, 
and is a local isometry mapping a Brownian motion on~$(\Ss,g)$ 
to a Brownian motion on the hyperbolic half-plane~$(\HH,h)$.
The refined analysis of Theorem~\ref{Th:MaintheoremSec2} confirms that we can treat the effect of the parameters~$\rho$ and~$\beta$ separately. 
Disassembling the influence of the parameters~$\rho$ and~$\beta$ 
further allows us to draw consequences on the large-time behaviour of these processes 
(see Remark~\ref{RemarktoSection3} and Section~\ref{SS:ltb} below). 
\end{remark}
As a first step we investigate the maps $\widehat \varphi^{0}, \widetilde \varphi^{0},$ 
which annul~$\beta$ and~$\chi, \overline{\chi}$, which reintroduce~$\beta$.
\begin{lemma}\label{Lem:varphi}
The solution~$(X,Y)$ to~\eqref{eq:SABRDrift} coincides in law 
with a Brownian motion on~$(\Ss,g)$. 
The process~$(\widehat{X},\widehat{Y})$ defined pathwise by
\begin{equation}\label{E:inad}
(\widehat{X}_t,\widehat{Y}_t)
 := \widehat \varphi^{0}(X_t, Y_t)
 = \left(\frac{X_t^{1-\beta}}{1-\beta},Y_t\right),\quad \text{for all }t\geq 0.
\end{equation}
is a SABR process~\eqref{eq:SABRSDE} with $\beta=0$,
which coincides in law with a Brownian motion
on the correlated hyperbolic plane~$(\Ss^0,g^0)$. 
\end{lemma}
\begin{remark}\label{RemarktoSection3}
This map $\widehat \varphi^{0}:\Ss\to\Ss^{0}$ will be essential in our study of the long-time behaviour of the process~$X$ 
in Section~\ref{SS:ltb}. 
Indeed, since $\beta \in [0,1)$, we have, for the probability defined in~\eqref{eq:PMass},
\begin{equation}\label{E:forsome}
\PPP
 := \PP(X_t=0\text{ for some }t \in (0,\infty))
 = \PP(\widehat{X}_t=0\text{ for some }t \in (0,\infty)).
\end{equation}
\end{remark}
\begin{proof}
The statement that~\eqref{eq:SABRDrift} has the same law as Brownian motion on~$(\Ss,g)$ 
follows from the computation of the infinitesimal generator of~\eqref{eq:SABRDrift}, 
which coincides with the Laplace-Beltrami operator~$\frac{1}{2}\Delta_g$ on a manifold 
with metric tensor~$g(x,y)$ (see~\eqref{eq:LaplaceBeltrSabr} and~\eqref{eq:GeneratorSabr}).
The second statement is an application of It\^o's formula, 
which transforms the system~\eqref{eq:SABRDrift} into 
\begin{equation}\label{eq:SABRdynamicsWithStratDrift2}
\begin{array}{rll}
\D \widehat{X}_t & := \widehat Y_t \D W_t, \qquad
 &\displaystyle  \widehat{X}_0 = \widehat{x}_0 : = x_0^{1-\beta} / (1-\beta),\\
\D \widehat Y_t & = \nu \widehat Y_t \D Z_t, \qquad & \widehat Y_0 = \widehat y_0,\\
\D \langle W,Z\rangle_t & = \rho \D t,
\end{array}
\end{equation}
which is identical to~\eqref{eq:SABRSDE} with $\beta=0$, and $\rho\in (-1,1)$. 
Its generator coincides with the Laplace-Beltrami operator~$\frac{1}{2}\Delta_{g^0}$ 
of the respective manifold, which yields the last statement. 
\end{proof}
\begin{lemma}\label{Lem:SABRIsoS0}
The map~$\chi$ (resp.~$\bar{\chi}$) is a local isometry between~$(\Ss^{0}_+,g^0)$ and~$(\Ss,g)$ 
(resp.~$(\HH_+,h)$ and~$(\UU,u)$) and transforms the Brownian motion 
on the hyperbolic plane~$(\Ss^{0}_+,g^0)$ (resp.~$(\HH_+,h)$), 
whose dynamics are described by~\eqref{eq:SABRdynamicsWithStratDrift2}, 
into a Brownian motion on the general SABR plane~$(\Ss,g)$ (resp.~$(\UU,u)$), 
satisfying~\eqref{eq:SABRDrift}.
\end{lemma}
\begin{proof}
For a local isometry 
between $(\Ss^{0}_+,g^0)$ and $(\Ss,g)$ (resp. $(\HH,h)$ and $(\UU,u)$), it holds that for any $(\widehat x,\widehat y)\in \Ss^0$ (resp. $(\widetilde{x}, \widetilde{y})\in \HH$) there exists a small open neighbourhood $U_{(\widehat x,\widehat y)}\subset \Ss^0$ (resp. $U_{(\widetilde{x}, \widetilde{y})}\subset \HH$), 
such that the map $\chi_{|_{U_{(\widehat x,\widehat y)}}}$ 
(resp. $\bar \chi_{|_{U_{(\widetilde{x}, \widetilde{y})}}}$) is an isometry onto its image;
in particular it satisfies the pullback relation
(pullback notations and definitions are explained in Appendix~\ref{app:Geometry})
$$
\left(\chi_* g\right)(x, y)
 = \chi_* \left( \frac{\D x^2}{\rrho x^{2\beta}y^2}+\frac{2 \rho  \D x \D y}{\rrho x^{\beta}y^2}+\frac{\D y^2}{\rrho y^2} \right)
 = \frac{\D\widehat{x}^2 +2 \rho  \D \widehat x \D \widehat y +\D\widehat{y}^2}{\rrho \widehat{y}^2}=g^0(\widehat{x}, \widehat{y}),
$$
respectively, for zero correlation
$$
\left(\bar \chi_* u\right)(\bar{x}, \bar{y})
 = \bar \chi_* \left( \frac{\D\bar{x}^2}{\bar{x}^{2\beta}\bar{y}^2}+\frac{\D\bar{y}^2}{\bar{y}^2} \right)
 = \frac{\D\widetilde{x}^2 + \D\widetilde{y}^2}{\widetilde{y}^2}=h(\widetilde{x}, \widetilde{y}).
$$
For any $(\widehat x,\widehat y)\in \Ss^0$ (resp. $(\widetilde{x},\widetilde{y})\in \HH$), the Jacobians read
$$
\nabla  \chi(\widehat{x},\widehat{y}) = 
\begin{pmatrix}
   (1-\beta)^{\frac{\beta}{1-\beta}} \widehat{x}^{\frac{\beta}{1-\beta}} & 0\\
  0 & 1
  \end{pmatrix}
\quad
\textrm{and}
\quad 
\nabla \bar \chi(\widetilde{x},\widetilde{y}) = 
\begin{pmatrix}
   (1-\beta)^{\frac{\beta}{1-\beta}} \widetilde{x}^{\frac{\beta}{1-\beta}} & 0\\
  0 & 1
  \end{pmatrix},
$$
respectively, hence the local pullback property 
is clearly satisfied by $\chi$ (resp. $\bar \chi$). 
The last statement follows by It\^{o}'s lemma.
\end{proof}
The maps $\widetilde \phi_{0}, \overline \phi_{0}$ affect the correlation parameter as follows:
\begin{lemma}\label{CorrHypBM}
The map $\widetilde \phi_{0}: \Ss^{0} \to \HH$ is a global isometry and transforms the SABR model~\eqref{eq:SABRSDE} with $\beta=0$ into a Brownian motion on~$(\HH,h)$.
Furthermore, the heat (or transition) kernel of the solution of the system~\eqref{eq:SABRdynamicsWithStratDrift2} is available in closed form:
$$
\rrho^{-1}K_{\phi_{0}(x,y)}^{h}(s,\phi_{0}(x,y)),\qquad \text{for }s>0\text{ and } (x,y)\in\Ss^{0},
$$
where $K_{(\widetilde{x}, \widetilde{y})}^{h}(s,\cdot)$ denotes the hyperbolic heat kernel 
at $(\widetilde{x}, \widetilde{y})\in \HH$, 
for which a closed-form expression and short- and large-time asymptotics 
are known (\cite[Equation (9.35)]{Grigoryan} and~\cite{HLW}).
\end{lemma}
\begin{proof}
The following shows that~$\phi_{0}$ is in fact a global isometry: 
$\widetilde \phi_{0}$ is onto and invertible on~$\Ss^{0}$ and, for any $(x,y)\in \Ss^{0}$, its Jacobian
$$
\nabla \widetilde \phi_{0}(x,y) = 
\begin{pmatrix}
1/\rrho & -\rho / \rrho\\ 
0 & 1
\end{pmatrix},
$$ 
is independent of~$x$ and does not explode at~$x=0$. 
Furthermore, for any $(x,y)\in\Ss^{0}$,
\begin{align*}
\left(\widetilde \phi_{0_*} h\right)(x,y)
 = & \widetilde \phi_{0_*}\left( \frac{\D\widetilde{x}^2 + \D\widetilde{y}^2}{\widetilde{y}^2} \right)
 = \frac{1}{y^2}\left(\frac{\D x}{\rrho} - \frac{\rho \D y}{\rrho}\right)^2 + \frac{(\D y)^2}{y^2} = g^{0}(x,y).
\end{align*}
The last statement follows from Lemma~\ref{LemmaKernelRelation} together with 
$\det(\nabla \widetilde \phi_{0}(\cdot)) = 1 / \rrho \ne 0$.
One can easily verify by It\^{o}'s lemma that the dynamics ~\eqref{eq:SABRdynamicsWithStratDrift2} for general $\rho \in (-1,1)$ are transformed into~\eqref{eq:SABRdynamicsWithStratDrift2} for $\rho=0$ under the map $\widetilde \phi_{0}$.
\end{proof}

We now verify that $\bar{\phi_0}$ is a `geometry-preserving' map from the general SABR plane~$(\Ss,g)$ 
into the uncorrelated SABR plane~$(\UU,u)$, which of course reduces to the identity map when $\rho=0$, and to $\widetilde{\phi}_0$ when $\beta=0$.
\begin{lemma}\label{Lem:SABRIsoU}
For any $(\rho,x) \in (-1,0]\times \Ss$, 
the map~ $\bar{\phi_0}$ 
is a local isometry between~$(\Ss,g)$ and~$(\UU,u)$.
\end{lemma}
\begin{proof}
The statement follows directly from the fact that the map $\bar{\phi}_{0}$ and its partial derivatives
\begin{align}\label{PartialDerivativesLocalIsometry3}
\begin{split}
  &\partial_x \bar{x}(x,y) = \frac{x^{-\beta}}{\rrho}(1-\beta)^{\frac{\beta}{1-\beta}} \left(\frac{x^{1-\beta}}{\rrho(1-\beta)}-\frac{\rho y}{\rrho}\right)^{\beta/(1-\beta)},\\
  &\partial_y \bar{x}(x,y) = -\frac{\rho}{\rrho}(1-\beta)^{\frac{\beta}{1-\beta}} \left(\frac{x^{1-\beta}}{\rrho(1-\beta)}-\frac{\rho y}{\rrho}C\right)^{\beta/(1-\beta)},\\
  &\partial_x \bar{y}(x,y) = 0,\qquad
    \partial_y \bar{y}(x,y) = 1,
\end{split}
\end{align}
satisfy the following system of partial differential equations implied by the local pullback property 
$\left(\bar{\phi}_{0}^* u\right)(\bar{x}, \bar{y}) = g(x,y)$,
for any $(x,y)\in \Ss$, $(\bar x, \bar y) \in \mathbb{U}$ for the Riemannian metrics $g$ and $u$:
\begin{equation*}
\left\{
\begin{array}{rl}
\displaystyle \frac{( \partial_x \bar{x})^2 }{\bar{x}^{2\beta}\bar{y}^2}+ \frac{(\partial_x \bar{y})^2}{\bar{y}^2}
 & \displaystyle = \frac{1}{\rrho^2y^2x^{2\beta}},\\ 
\displaystyle \frac{2(\partial_x \bar{x}\partial_y \bar{x}) }{\bar{x}^{2\beta}\bar{y}^2}+ \frac{2(\partial_x \bar{y}\partial_y \bar{y}) }{\bar{y}^2}
 &\displaystyle =
\frac{-2 \rho}{\rrho^2y^2 x^{\beta} },\\
\displaystyle \frac{(\partial_y \bar{x} )^2}{\bar{x}^{2\beta}\bar{y}^2}+ \frac{(\partial_y \bar{y})^2}{\bar{y}^2}
 & \displaystyle = \frac{1}{\rrho^2y^2}.
\end{array}
\right.
\end{equation*}
\end{proof}
As an application of Lemma \ref{Lem:SABRIsoU} it may be possible to relate the absolutely continuous part 
of the distribution of the Brownian motion on the uncorrelated SABR plane~$(\UU,u)$ and
that of the Brownian motion on the general SABR plane~$(\Ss,g)$ 
via the relation~\eqref{KernelRelation} of the heat kernels~\cite{SurfaceMeasuresforBrownianMotion};
this can be performed following similar steps as in~\cite{HLW}, but care is needed, as discussed below.

\begin{lemma}\label{lem:KgKu}
Let $K^g_Z$ and $K^u_Z$ denote the fundamental solutions\footnote{More details about these fundamental solutions can be found on Page~\pageref{page:FundamSol}.}
at~$Z\in \Hh_+$ (i.e. the limit $\lim_{s\downarrow 0} K_Z^g(s,\cdot)=\delta_Z(\cdot)$ is the Dirac delta distribution),  of the heat equations corresponding to the metrics $g$ and $u$.
Then, for any $z=(x,y)\in\Hh_+$,
\begin{equation}\label{KernelRelation}
 K_Z^g(s,z)
 = \frac{(1-\beta)^{\frac{\beta}{1-\beta}}}{\rrho x^{\beta}} \left(\frac{x^{1-\beta}}{\rrho(1-\beta)}-\frac{\rho y}{\rrho}\right)^{\frac{\beta}{1-\beta}}\ K_{\bar{\phi}_{0}(z)}^u(s,\phi_{0}^{0}(z)).
\end{equation}
When $\beta = 1/2$, the formulae simplify to
$\bar{\phi}_{0}(x,y) \equiv \left( \frac{1}{(1-\rho)^2 }\left( x-\sqrt{x}\rho y+\frac{\rho^2 y^2}{4}\right),y \right)$, and
$\det(\nabla \bar{\phi}_{0}(x,y)) = \left( 1-\frac{\rho y }{2 \sqrt{x}}\right) / (1-\rho)^2$,
for all $(x,y)\in\Ss$.
\end{lemma}
\begin{proof}
The statement follows from Lemma~\ref{LemmaKernelRelation}:
the Radon-Nikodym derivatives are 
$\frac{\D z}{\D\mu_g(z)}=\rrho^2y^2x^{\beta}$ and $\frac{\D\bar{z}}{\D\mu_u(\bar{z})}=\bar{y}^2\bar{x}^{\beta}$,
with $\mu_g$ and $\mu_u$ the Riemannian volume elements 
on~$\Ss$ and~$\UU$ (Definition~\ref{def:ManifoldFundamenalSol} in Appendix~\ref{app:Geometry}),
and the Jacobian of $\bar{\phi}_{0}$ at $z=(x,y)\in \Ss$ is as in~\eqref{PartialDerivativesLocalIsometry3},
so that
$$
\det\left(\nabla \bar{\phi}_{0}(x,y)\right)
 = \frac{(1-\beta)^{\frac{\beta}{1-\beta}}}{\rrho x^{\beta}} \left(\frac{x^{1-\beta}}{\rrho(1-\beta)}-\frac{\rho y}{\rrho}\right)^{\frac{\beta}{1-\beta}}.
$$
\end{proof}
\label{rem:SingularDeterminant}
Such a relation of heat kernels relies on the commutativity property of Laplace-Beltrami operators in Lemma~\ref{lem:CommLaplBeltr}, 
which is not meaningful for~\eqref{eq:LaplaceBeltrSabr} at~$x=0$ for general~$\beta$. 
Hence a statement relating the heat kernels might not hold true in the vicinity of the origin. 
Although in the case of exploding Jacobians the relation~\eqref{KernelRelation} 
of `kernels' formally indicates that the map under consideration induces an atom, 
it does not allow for an exact computation.
All the maps introduced above, except~$\bar{\phi}_{0}$ and~$\widetilde\phi_{0}^{0}$, are defined on 
$(0,\infty)^2$ and can be extended by continuity along the ray $\{0\}\times (0,\infty)$; 
they are however not differentiable there.
A direct application of It\^o's lemma would therefore fail, and an enhanced version would be needed, 
which in turn could (alluding to It\^o-Tanaka-Meyer-type formulae~\cite[Theorem 1.5, Chapter VI.1]{RY}) induce local times there.
However, we bypass this issue by imposing Dirichlet (absorbing) boundary conditions along this particular ray.
A statement similar to Lemma~\ref{lem:KgKu} below was made in~\cite{HLW} 
relating~$K^g$ to the hyperbolic heat kernel~$K^h$;
in their analysis, the determinant was $\det (\nabla \widetilde \phi_0^0(x,y)) \equiv x^{-\beta} / \rho$.

The knowledge of the exact form of the absolutely continuous part of the distribution would provide a means to infer the probability of the process~$X$ hitting its boundary.
This, however, would involve intricate formulae with multiple integrals. 
Instead, in the following section, we compute this probability in a more concise way, 
and use the knowledge of the kernel only to show that introducing first~$\beta$ then~$\rho$ (or conversely)
has no influence on this probability.

\section{Probability of hitting the boundary}\label{SS:ltb}
Having characterised the isometries between the Brownian motion on the hyperbolic plane
and a more general version, with drift, on the SABR plane, 
we derive here a concise formula to compute the hitting probability ~$\PPP$ (in~\eqref{eq:PMass}) 
of the boundary of this general process.
A key ingredient here is to note that this probability is equal to the limit of $\PP(X_t=0)$ as~$t$ tends to infinity.
We shall also determine the influence of the model parameters $(\beta,\rho)$ on this quantity.
The computation of this probability (Theorem~\ref{thm:Main} below) 
follows the works of Hobson~\cite{Hobson} on time changes.
We apply such a technique to progress from the Brownian motion on the correlated hyperbolic plane \eqref{eq:SABRdynamicsWithStratDrift2}
to a correlated Brownian motion on the Euclidean plane.
The joint distribution of hitting times of zero of two (correlated) Brownian motions without drift
was first established by Iyengar~\cite{Iyengar}, and refined by Metzler~\cite{Metzler}
(see also~\cite{BranchingModel} for further results on hitting times of correlated Brownian motions).
We also borrow some ideas from~\cite{DoeringHorvathTeichmann}, 
where Hobson's construction for the normal SABR model~\cite[Example 5.2]{Hobson} 
is extended to~\eqref{eq:SABRDrift} for general $\beta\in [0,1]$. 
This indeed follows from the observation that stochastic time change methods, 
going back to Volkonskii~\cite{Volkonski}, can still be applied to the Brownian motion on the SABR plane.
In order to formulate our next statement, we introduce several auxiliary parameters (see~\cite{Metzler}): 
$$
a_1 := \frac{x_0^{1-\beta}}{1-\beta},
\qquad 
a_2 := \frac{y_0}{\nu},
\qquad
r_0 := \sqrt{\frac{a_1^2+a_2^2-2\rho a_1a_2}{\rrho^2}},
$$
\begin{equation*}
\alpha := \left\{ \begin{array}{lll}
\displaystyle \pi + \arctan\left(-\frac{\rrho}{\rho}\right)	, & \mbox{if $\rho> 0$},\\
\displaystyle \frac{\pi}{2}, & \mbox{if $\rho=0$},\\
\displaystyle \arctan\left(-\frac{\rrho}{\rho}\right), & \mbox{if $\rho< 0$},
\end{array}
\right.
\qquad\qquad
\theta_0 := \left\{ \begin{array}{lll}
\displaystyle \pi + \arctan\left(\frac{a_{2}\rrho}{\rho}\right), & \mbox{if $a_1<\rho a_2$},\\
\displaystyle \frac{\pi}{2}, & \mbox{if $a_1=\rho a_2$},\\
\displaystyle \arctan\left(\frac{a_{2}\rrho}{\rho}\right), & \mbox{if $a_1>\rho a_2$}.
\end{array}
\right.
\end{equation*}
\begin{theorem}\label{thm:Main}
For the SDE~\eqref{eq:SABRDrift}, 
the probability~\eqref{eq:PMass} satisfies
$$
\PPP = \int_0^{\infty}\D t\int_0^t f(s,t)\D s,
$$
where for any $s<t$ 
($\mathrm{I}_{z}$ denotes the modified Bessel function of the first kind~\cite[Page 638]{Borodin}),
\begin{align*}
f(s, t) &= \frac{\pi\sin(\alpha)}{2\alpha^2 (t-s)\sqrt{s(t-s \cos^2(\alpha))}}\exp\left(-\frac{r_0^2}{2s}\frac{t-s\cos(2\alpha)}{2t-s(1+\cos(2\alpha))}\right) \\
&\quad\times\sum_{n=1}^{\infty}n\sin\left(\frac{n\pi(\alpha-\theta_0)}{\alpha}\right)
\mathrm{I}_{\frac{n\pi}{2\alpha}}
\left(\frac{r_0^2}{2s}\frac{t-s}{2t-s(1+\cos(2\alpha))}\right).
\end{align*}
\end{theorem}
\begin{remark}\label{R:rem1}
In the uncorrelated case $\rho=0$, the expressions in Theorem~\ref{thm:Main} simplify to
$\alpha = \frac{\pi}{2}$, $\theta_0 = \arctan\left(\frac{a_2}{a_1}\right)$, $r_0 = \sqrt{a_1^2 + a_2^2}$, and
$$
f(s, t) = \frac{2}{\pi (t-s)\sqrt{st}}
\exp\left(-\frac{r_0^2(t+s)}{4st}\right)
\sum_{n=1}^{\infty}n\sin\left(2n\left(\frac{\pi}{2}-\theta_0\right)\right)\mathrm{I}_{n}
\left(\frac{r_0^2(t-s)}{4st}\right).
$$
\end{remark}
\begin{remark}
When $\beta=0$, in Theorem~\ref{thm:Main} above, $a_1$ is equal to the starting point~$x_0$.
In this case~\eqref{eq:SABRDrift} corresponds
to the original `normal' SABR model~\eqref{eq:SABRSDE} for any $\rho \in (-1,1)$.
\end{remark}
\begin{proof}
Recalling the process $\widehat{X}$ in~\eqref{E:inad}, and the SDE~\eqref{eq:SABRdynamicsWithStratDrift2},
we wish to apply~\cite[Theorem 3.1]{Hobson} to~\eqref{eq:SABRdynamicsWithStratDrift2}. Consider
the system of SDEs
\begin{equation}\label{eq:SABRdynamicsWithStratDrift3}
\begin{array}{rll}
\D \widetilde{X}_t & =\D \widetilde{W}_t, \qquad & \widetilde{X}_0 = \widehat{x}_0,\\
\D \widetilde{Y}_t & = \nu \D \widetilde{Z}_t, \qquad & \widetilde{Y}_0 = y_0,\\
\D \langle \widetilde{W},\widetilde{Z}\rangle_t & = \rho \D t,
\end{array}
\end{equation}
where $(\widetilde{W},\widetilde{Z})$ is a two-dimensional correlated Brownian motion. 
With the time-change process
\begin{align}\label{StoppingTime}
\tau(t) := \inf\left\{ u\geq 0: \int_0^u \wY_s^{-2}\D s\geq t \right\},
\end{align}
Theorem 3.1 in~\cite{Hobson} implies that 
\begin{equation}\label{E:arr}
\widehat{X}_t = \widetilde{X}_{\tau(t)}
\qquad\mbox{and}\qquad 
Y_t = \widetilde{Y}_{\tau(t)},
\end{equation}
for all $t\geq 0$. 
In addition, the map $\widehat \varphi^{0}$ in~\eqref{E:inad} gives 
$\displaystyle X_t=\left(x_0^{1-\beta} + (1-\beta)\widetilde{W}_{\tau(t)}\right)^{1/(1-\beta)}$
for all $t\geq 0$.
Let now $\varepsilon$ denote the explosion time of~\eqref{eq:SABRdynamicsWithStratDrift3}, 
namely the first time that either~$\widetilde{X}$ or~$\widetilde{Y}$ hits zero. 
It is also the first time that the process~$\widetilde{W}$ hits the level $-\widehat{x}_0$ 
or that~$\widetilde{Z}$ hits~$-y_0 / \nu$.
Set 
$\Gamma_t := \int_0^t \wY_s^{-2}\D s$
and 
$\zeta := \lim_{t\uparrow\varepsilon}\Gamma_t$.
The process~$\Gamma$ is strictly increasing and continuous, 
so that its inverse~$\Gamma^{-1}$ is well defined, 
and clearly the time-change process~\eqref{StoppingTime} satisfies $\tau=\Gamma^{-1}$. 
Consider a new filtration~$\mathcal G$ and two processes $W$ and $Z$ defined, for each $t\geq 0$, by
$\mathcal G_t := \mathcal F_{\tau(t)}$,
$$
W_t := \int_0^{\tau(t)}\frac{\D\widetilde{W}_s}{\widetilde{Y}_s}\D s
\qquad\mbox{and}\qquad
Z_t := \int_0^{\tau(t)}\frac{\D\widetilde{Z}_s}{\widetilde{Y}_s}\D s.
$$ 
Up to time~$\zeta$, $W$ and~$Z$ are $\mathcal G$-adapted Brownian motions, 
and the system $(W,Z,\widehat{X},Y)$ is a weak solution to (\ref{eq:SABRdynamicsWithStratDrift2}).
It is therefore clear that 
$
\PP\left(\tau^{\widetilde{X}}_0 \in \D s, \tau^{\widetilde{Y}}_0 \in \D t\right)
=\PP\left(\tau^{\widetilde{W}}_{-\widehat{x}_0} \in \D s, \tau^{\widetilde{Z}}_{-y_0/\nu} \in \D t\right).
$
Moreover, it follows from \cite[Equation 3.2]{Metzler} with $\vec{\mu}=\vec{0}$, 
$\vec{x}_0=(\widehat{x}_0,y_0)$, and $$
\sigma=\left( \begin{array}{cc}
\rrho & \rho \\
0 & \nu \\
\end{array} \right),
$$    
that
$\PP\left(\tau^{\widetilde{X}}_0 \in \D s, \tau^{\widetilde{Y}}_0 \in \D t\right) = f(s, t) \D s \D t$,
where the function $f$ is defined in Theorem~\ref{thm:Main}, so that
\begin{equation}
\PP\left(\tau_0^{\widetilde{X}} < \tau_0^{\widetilde{Y}}\right) = \int_{0}^{\infty}\D t\int_{0}^{t}f(s,t)\D s.
\label{E:itf3}
\end{equation}

Reversing the arguments presented in~\cite{DoeringHorvathTeichmann, Hobson}, 
the probability $\PP(\tau_0^{\widetilde{X}} < \tau_0^{\widetilde{Y}})$ 
coincides with the probability that the process $\widehat{X}$ hits zero over the time horizon $[0,\infty)$. 
Indeed, through~\eqref{E:arr}, the time change~\eqref{StoppingTime} 
converts the Brownian motion~$\widetilde{Y}$ into a geometric Brownian motion~$Y$ started at~$y_0>0$, 
so that the (a.s. finite) point $\tau^{\widetilde{Y}}_0$ is mapped to $\tau_0^Y = \infty$. 
Therefore the time-changed process $\widetilde{X}$ over~$[0,\tau^{\widetilde{Y}}_0)$ 
corresponds to~$\widehat{X}$ considered over~$[0,\infty)$ and, using~\eqref{E:arr}, we obtain
$$
\PP\left(\tau_0^{\widetilde{X}} < \tau_0^{\widetilde{Y}}\right)
 = \PP\left(\tau_0^{\widehat{X}} < \tau_0^Y\right)
 = \PP\left(\tau_0^{\widehat{X}} < \infty\right)
 = \PP\left(\widehat{X}_t=0, \text{ for some } t\in (0,\infty)\right),
$$
and Theorem~\ref{thm:Main} follows from~\eqref{E:forsome} and~\eqref{E:itf3}.
\end{proof}

\begin{remark}\label{R:interior}
For the normal SABR model ($\beta=0$) in~\eqref{eq:SABRSDE}, 
Hobson~\cite[Example 5.2]{Hobson} found the following formula for the price process $X$: 
$$
X_t=\frac{\rho}{\nu}\left(\widetilde {Y}_{\tau(t)} -y_0\right)+\rrho^2\widetilde{Z}_{\tau(t)},
\qquad\text{for all }t\geq 0,
$$
where the process~$\widetilde{Y}$ and the Brownian motion~$\widetilde{Z}$ 
are the same as in~\eqref{eq:SABRdynamicsWithStratDrift3}, and $\tau$ is as in~\eqref{StoppingTime}.
\end{remark}

\begin{remark}\label{R:rem0}
For $\beta=1$, the SDEs~\eqref{eq:SABRSDE} and~\eqref{eq:SABRDrift} read
$$
\D X_t = X_t Y_t\D W_t
\qquad\text{and}\qquad
\D X_t = X_t\left(Y_t \D W_t +\frac{1}{2}Y_t^2\D t\right),
$$
respectively, and, by the Dol\'{e}ans-Dade formula~\cite[Section IX-2]{RY},
the solutions to these equations are exponential functionals, 
and therefore do not exhibit mass at the origin.
\end{remark}


\appendix 
\section{Reminder on the heat equation on manifolds}\label{app:Geometry}
We recall some standard results on heat kernels on Riemannian manifolds, 
needed in Section~\ref{sec:Mappings}.
For a given metric~$g$, we denote by~$\Delta_g$ the corresponding Laplace-Beltrami operator.
Following the notations from~\cite[Section 3.12]{Grigoryan}, 
let $k \in \mathbb{N} \cup \{\infty\}$, 
$M_1, M_2$ two $C^{k+2}$-manifolds 
and $\phi:M_2 \to M_1$ a $C^{k+2}$-diffeomorphism which is
an isometry between $(M_2,g_2)$ and $(M_1,g_1)$. 
Any function~$f$ on~$M_1$ induces a pullback function~$\phi_* f$ on~$M_2$
by the relation $\phi_* f = f\circ\phi$.
We start with a fundamental property of this operator (\cite[Lemma 3.27]{Grigoryan}).

\begin{lemma}\label{lem:CommLaplBeltr}
The Laplace-Beltrami operator~$\Delta_{g_i}$ ($i=1,2$) commutes with~$\phi$ 
in the sense that $\Delta_{g_2}(\phi_* f)= \phi_*(\Delta_{g_1} f)$
holds for any $f \in C^{k+2}(M_1)$.
\end{lemma}
\begin{definition}\label{def:ManifoldFundamenalSol}
Let $(M,g)$ be a smooth Riemannian manifold and $Z\in M$.
The smooth function
$p_{Z}:  (0,\infty)\times M \to\mathbb{R}$
is a fundamental solution at $Z$ of the heat equation on $(M,g)$ if:
\begin{enumerate}
\item[(i)] it solves the heat equation $\Delta_g p_Z = \partial_{t} p_Z$ on~$(M,g)$;
\item [(ii)] $\lim_{t\downarrow 0}p_Z(t,\cdot) = \delta_Z(\cdot)$, 
where $\delta_Z$ denotes the Dirac measure at~$Z\in M$:
$$
\lim_{t\downarrow 0}\int_M p_Z(t,z) f(z) \mu_g(\D z)=f(Z), 
$$
for all test functions 
$f \in C^{\infty}_0(M)$, with $\mu_g(\D z)$ being the Riemannian volume element at~$z$.
\end{enumerate}
The fundamental function $p_Z$ is said to be regular if furthermore $p_Z\geq0$ and
$\int_M p_Z(t,z)\mu_g(\D z)  \leq 1$.
\end{definition}

\begin{proposition}\label{RelationFundamentalSol}
Let $k \in \mathbb{N}\cup\{0\} \cup\{\infty\}$, $\phi:(M_2,g_2)\to(M_1,g_1)$ a $C^{k+2}$-smooth isometry, 
$p^{g_1}_{Z_1}$ the fundamental solution at $Z_1\in M_1$ of the heat equation on $(M_1,g_1)$,
and let $Z_2\in M_2$ be such that $\phi(Z_2)=Z_1$. 
Then the map
$(t, z_2)\mapsto p^{g_1}_{\phi(Z_2)}(t,\phi(z_2)) \equiv \phi_*p^{g_1}_{Z_1}(t,z_1)$
is the (unique) fundamental solution at~$Z_2$ of the heat equation on $(M_2,g_2)$.
\end{proposition}
\begin{proof}
Lemma~\ref{lem:CommLaplBeltr} implies that Definition~\ref{def:ManifoldFundamenalSol}(i) holds for the above map.
The operator~$\Delta_{g_2}$ acts only on the space variable $z_2 \in M_2$ 
and not on the fixed point $Z_2 \in M_2$, so that
\begin{align}\label{FundSolIsomInvariant}
\Delta_{g_2} p^{g_1}_{\phi(Z_2)}(t,\phi(z_2))
 & = \Delta_{g_2} \left(\phi_*p^{g_1}_{Z_1}(t,z_1)\right)
 = \phi_*\left( \Delta_{g_1}p^{g_1}_{Z_1}(t,z_1) \right)
= \phi_*\left( \frac{\partial }{\partial t}p^{g_1}_{Z_1}(t,z_1) \right)\nonumber\\
 & = \frac{\partial }{\partial t} p^{g_1}_{\phi(Z_2)}(t,\phi(z_2)),
\end{align}
with $z_1:=\phi(z_2)$, $Z_1:=\phi(Z_2)$, where the first equality follows from the pullback relation,
the second from the commutativity relation in Lemma~\ref{lem:CommLaplBeltr}, 
and the third one since~$p^{g_1}_{Z_1}$ satisfies the heat equation on~$(M_1,g_1)$.
We now check Definition~\ref{def:ManifoldFundamenalSol}(ii).
Let $f_{1} \in C_0^{\infty}(M_1)$ be a test function and $f:= \phi_* f_{1}$.
Set $z_{1} = \phi(z_2)$ and $Z_{1} = \phi(Z_2)$ for any $z_2, Z_2 \in M_2$. 
Given that~$\phi$ is an isometry, 
so is~$\phi^{-1}$ and the pullback $(\phi^{-1})_*\mu_{g_2}(\D\cdot)$ coincides with 
the volume form on $(M_1,g_1)$.
Then
\begin{align*}
\lim_{t\downarrow 0} \int_{M_2}
 p^{g_1}_{\phi(Z_2)}(t,\phi(z_2)) f(z_2) \mu_{g_2}(\D z_2)
=& \lim_{t\downarrow 0} \int_{M_2}
 p^{g_1}_{Z_1}(t,\phi(z_2)) f_{1}(\phi(z_2)) \mu_{g_2}(\D z_2)\\
=& \lim_{t\downarrow 0} \int_{M_1}
 p^{g_1}_{Z_{1}}(t,z_{1}) f_{1}(z_{1})  \left((\phi^{-1})_*\mu_{g_2}\right)(\D z_{1})\\
=& \lim_{t\downarrow 0} \int_{M_1}
 p^{g_1}_{Z_{1}}(t,z_{1}) f_{1}(z_{1}) \mu_{g_1}(\D z_{1})
 =  f_{1}(Z_{1})
 = f\circ\phi(Z_2).
\end{align*}
\end{proof}
The fundamental solutions in Proposition~\ref{RelationFundamentalSol} 
are denoted with respect to the Riemannian volume form of the respective manifold,
whereas they are expressed in terms of the Lebesgue measure (the volume form on the Euclidean plane) in Lemma~\ref{lem:KgKu} 
and Lemma~\ref{LemmaKernelRelation}.
\label{page:FundamSol}This translation can be performed as follows:
let the Riemannian volume form be given in orthogonal coordinates, 
and let $K^g_Z$ denote the fundamental solutions (in terms of the Lebesgue measure) 
at $Z\in \Hh_+$ of the heat equation corresponding to the Riemannian metric~$g$ 
in the sense that the Radon-Nikodym derivative with respect to the Lebesgue measure
is already incorporated into the expression for~$K^g_Z$:
if $p_{Z}^g(s,\cdot)$ denotes the fundamental solution (at $Z\in \Hh_+$) as in 
Proposition~\ref{RelationFundamentalSol}, then, for any test function~$f$, 
$$
 \int_{\Hh_+}  f(z) K_{Z}^g(s,z) \D z                                     
 := \int_{\Hh_+}  f(z) p^g_{Z}(s,z)\frac{\D z}{\mu_g(\D z)} \mu_g(\D z) 
 = \int_{\Hh_+}  f(z) p^g_{Z}(s,z) \frac{\mu_g(\D z)}{\sqrt{\det(g)}}.
$$
The following lemma follows directly from Proposition~\ref{RelationFundamentalSol}.
In order to translate the coordinate-free result of Proposition~\ref{RelationFundamentalSol} 
to our setting, we assume from now on that $M_1=M_2 = \Hh_+$.
\begin{lemma}\label{LemmaKernelRelation}
For any $i=1,2$, let $K^{g_i}_{Z_i}$ denotes the fundamental solution at $Z_i\in \Hh_+$of the heat equations corresponding to the metric $g_i$:
\begin{equation}\label{HeatEquations}
\left\{
\begin{array}{ll}
 & \displaystyle \partial_{s} K^{g_i}_{Z_i} = \frac{1}{2} \Delta_{g_i} K^{g_i}_{Z_i},\\
 & K^{g_i}_{Z_i}(0,z_i)=\delta(z_i - Z_i).
\end{array}
\right.
\end{equation}
If~$\phi:(\Hh_+, g_2) \to (\Hh_+, g_1)$ is an isometry such that $\phi(Z_2) = Z_1$ and $\phi(z_2) = z_1$, then 
\begin{align}\label{KernelRelation}
K_{Z_1}^{g_1}(s,z_1) = \det\left(\nabla \phi(Z_2)\right) K_{\phi(Z_2)}^{g_2}(s,\phi(z_2)).
\end{align}
\end{lemma}
The generators of the Brownian motions on $(\Ss,g)$ 
(resp. $(\UU,u)$)--(defined in Section~\ref{sec:Mappings})--are defined on their respective spaces 
with $\{x\neq0\}$ and $\{\bar{x}\neq0\}$ for $\beta\neq0$
respectively and read
\begin{equation}\label{eq:LaplaceBeltrSabr}
\begin{array}{rll}
 \Delta_g f & = 
\displaystyle y^2 \left(\beta x^{2\beta-1} \frac{\partial f}{\partial x} +  
   x^{2\beta} \frac{\partial^2 f}{\partial x^2} + 2 \rho x^{\beta} \frac{\partial}{\partial x}\frac{\partial f}{\partial y}
  + \frac{\partial^2 f}{\partial y^2} \right), & \text{for any }f \in C^{k+2}(\Ss),\\
 \Delta_u f & = 
\displaystyle \bar{y}^2 \left(\beta\bar{x}^{2\beta-1} \frac{\partial f}{\partial \bar{x}} +  
   \bar{x}^{2\beta} \frac{\partial^2 f}{\partial \bar{x}^2} + \frac{\partial^2 f}{\partial \bar{y}^2} \right),
    & \text{for any }f \in C^{k+2}(\UU),
\end{array}
\end{equation} 
while the infinitesimal generators of the original SABR model~\eqref{eq:SABRSDE} are
\begin{equation}\label{eq:GeneratorSabr}
\begin{array}{rll}
\mathcal{A}f & = 
\displaystyle y^2 \left(  
   x^{2\beta} \frac{\partial^2 f}{\partial x^2}+2 \rho x^{\beta} \frac{\partial}{\partial x}\frac{\partial f}{\partial y} + \frac{\partial^2 f}{\partial \bar{y}^2} \right),
    & \text{for any }f \in C^{k+2}(\Ss),\\
\mathcal{A}_{\rho=0}f & = 
\displaystyle \bar{y}^2 \left(  
   \bar{x}^{2\beta} \frac{\partial^2 f}{\partial \bar{x}^2} + \frac{\partial^2 f}{\partial \bar{y}^2} \right),
    & \text{for any }f \in C^{k+2}(\UU),
\end{array}
\end{equation} 
Note that for $\beta=0$ the operators $\Delta_g$ and $\mathcal{A}$ (resp. $\Delta_u$ and $\mathcal{A}_{\rho=0}$) coincide.

\end{document}